\documentclass[10pt,reqno]{amsart}

\usepackage{amssymb} 
\usepackage[english, spanish]{babel}
\usepackage[utf8]{inputenc}
\usepackage{cancel}
\usepackage{multicol}
\usepackage{amsmath}
\usepackage{url}
\usepackage{hyperref}

\numberwithin{equation}{section}

\newtheorem{Theorem}{\sc Theorem}[section]
\newtheorem{Lemma}{\sc Lemma}[section]
\newtheorem{Proposition}{\sc Proposition}[section]

\newtheorem{Definition}{\sc Definition}[section]

\title[Some Existence Results on Cantor Sets]{Some Existence Results on Cantor Sets}

\author[B. \'Alvarez-Samaniego, W. P. \'Alvarez-Samaniego, J. Ortiz-Castro]{}

\date{October 31, 2016}

\email{balvarez@uce.edu.ec, balvarez@impa.br}
\email{alvarezwilson@hotmail.com} 
\email{jonathan.ortizc@epn.edu.ec} 

\begin{document}
\maketitle 

\centerline{\scshape Borys \'Alvarez-Samaniego}
{\footnotesize
 \centerline{N\'ucleo de Investigadores Cient\'{\i}ficos}
   \centerline{Facultad de Ingenier\'{\i}a, Ciencias F\'{\i}sicas y Matem\'atica}
   \centerline{Universidad Central del Ecuador (UCE)}
   \centerline{Quito, Ecuador}}

\vspace{0.5cm}	
	
\centerline{\scshape Wilson P. \'Alvarez-Samaniego}
{\footnotesize
 \centerline{N\'ucleo de Investigadores Cient\'{\i}ficos}
   \centerline{Facultad de Ingenier\'{\i}a, Ciencias F\'{\i}sicas y Matem\'atica}
   \centerline{Universidad Central del Ecuador (UCE)}
   \centerline{Quito, Ecuador}}

\vspace{0.5cm}

\centerline{\scshape Jonathan Ortiz-Castro}
{\footnotesize
   \centerline{Facultad de Ciencias}
   \centerline{Escuela Polit\'ecnica Nacional (EPN)}
   \centerline{Quito, Ecuador}}


\vspace{1cm}
\selectlanguage{english}

\begin{abstract}{}
The existence of two different Cantor sets, one of them contained in 
the set of Liouville numbers and the other one inside the set of Diophantine 
numbers, is proved.  Finally, a necessary and sufficient condition for the existence 
of a Cantor set contained in a subset of the real line is given. 
\hfill \break \vspace{ 4 mm}\break {\it {MSC:}} 54A05; 54B05
\hfill \break \vspace{ 4 mm}\break {\it {Keywords:}} Cantor set; Liouville numbers; 
Diophantine numbers.
\end{abstract}

\vspace{1cm}

\maketitle

\section{Introduction}
First, we will introduce some basic topological concepts.
\begin{Definition}
A nowhere dense set $X$ in a topological space is a set whose closure has 
empty interior, i.e. $\text{int}(\overline{X})=\emptyset$.
\end{Definition}
\begin{Definition}
A nonempty set $C\subset \mathbb{R}$ is a Cantor set if $C$ is nowhere dense and perfect 
(i.e. $C=C'$, where $C':= \{p \in \mathbb{R}; \; p \text{ is an 
accumulation point of } C \}$ is the derived set of $C$).
\end{Definition}
\begin{Definition}
A condensation point $t$ of a subset $A$ of a topological space, is any point $t$, 
such that every open neighborhood of $t$ contains uncountably many points of $A$.
\end{Definition}
We denote by $\mathbb{Q}$ the set of rational numbers. The symbol $\mathbb{Z}$ 
is used to denote the set of integers, $\mathbb{Z}^+=\{1,2, \ldots \}$ denotes 
the positive integers and $\mathbb{N}=\{0,1,2, \ldots \}$ is the set of all natural 
numbers.  The cardinality of a set $B$ is denoted by $|B|$ .  
We denote by $\mathbf{OR}$, the class of all ordinal numbers.  Moreover, 
$\Omega$ represents the set of all countable ordinal numbers. 
\begin{Definition}\label{Liouville}
Let $\alpha \in \mathbb{R}\setminus \mathbb{Q}$. We say that $\alpha$ is a 
Liouville number if for all $n\in \mathbb{N}$, there exist 
integer numbers $p=p_n$ and $q=q_n$, such that $q>1$ and
\begin{equation}\label{Liouville_eq}
0<\left|\alpha-\frac{p}{q}\right|<\frac{1}{q^n}.
\end{equation}
If $\beta \in \mathbb{R}$ is not a Liouville number, we say that $\beta$ is a 
Diophantine number. The sets of Liouville and Diophantine numbers are, respectively,  
denoted by $\mathbb{L}$ and $\mathbb{D}$.
\end{Definition}
Joseph Liouville, by giving two different proofs, showed the existence of 
transcendental numbers for the first time in 1844 
(\cite{Liouville1844a, Liouville1844b}).  Later, in 1851 
(\cite{Liouville1851}) more detailed versions of these proofs were given.
It was also shown in \cite{Liouville1851} that the real number  
 \begin{equation*}
   \sum_{k=1}^{+\infty}\frac{1}{10^{k!}},
 \end{equation*}
which was already mentioned in \cite{Liouville1844a}, is a transcendental number. 
It is worth noting that the techniques used in 
\cite{Liouville1844a, Liouville1844b, Liouville1851}  
allow proving that all Liouville numbers are transcendental.

Some general properties of the set of Liouville numbers are:  $\mathbb{L}$ 
is a null set under the Lebesgue measure (i.e. $\lambda(\mathbb{L})=0$), it is a 
dense $G_{\delta}$ set in the real line, $\mathbb{L}$ is an uncountable set and, more 
specifically, it has the cardinality of the continuum.  Since the Lebesgue measure 
of the Diophantine numbers is infinity, it is an uncountable set.  From the fact that 
there exists a Cantor set in $\mathbb{D}$,  and since any 
uncountable closed set in $\mathbb{R}$ has the cardinality of the continuum, it 
follows that $\mathbb{D}$ has also the cardinality of the continuum. Moreover, in 
view of the density of $\mathbb{Q}$ in $\mathbb{R}$, the set of Diophantine numbers 
is also dense in the real line. In addition, $\mathbb{D}$ is a set of first category, 
i.e. it can be written as a countable union of nowhere dense subsets of $\mathbb{R}$.

Now, we would like to prove a basic result related to the fact that the property of 
being a Cantor set is preserved by homeomorphisms when the homeomorphic image of 
its domain is a closed subset of $\mathbb{R}$. 
\begin{Proposition}\label{P0}
Let $A,B\subset \mathbb{R}$.  Suppose that $f:A \longmapsto B$ is a homeomorphism
and $B$ is a closed subset of $\mathbb{R}$.  If $C\subset A$ is a Cantor set, 
then $f(C)\subset B$ is also a Cantor set.
\end{Proposition}
\begin{proof}
From the fact that $C \subset A$ is a perfect set, we get that $C$ is 
closed subset of $\mathbb{R}$.  Moreover, since $f^{-1}$ is continuous, 
we have that $(f^{-1})^{-1}(C)=f(C)$ is a closed set in $B$, and since 
$B$ is a closed subset of $\mathbb{R}$, it follows that $f(C)$ is closed in 
$\mathbb{R}$.  Now, we claim that $f(C)\subset (f(C))'$.  In fact, 
let $y\in f(C)$ and $\varepsilon>0$. So, there is $x\in C$ such that $y=f(x)$. 
Since $f$ is continuous at $x$, there exists $\delta>0$ such that
\begin{equation}\label{cont}
 \text{for all }z\in A, \quad |z-x|<\delta\Longrightarrow |f(z)-f(x)|<\varepsilon.
\end{equation}
Considering that $x \in C=C'$, there exists $z_0 \in C\subset A$ such that
$0<|z_0-x|<\delta$, and by (\ref{cont}) we deduce that $|f(z_0)-f(x)|<\varepsilon$.  
Furthermore, the injectivity of $f$ implies that $f(x)\neq f(z_0)$.  Then,  
$f(z_0)\in \left(B(y,\varepsilon)\setminus\{y\}\right)\cap f(C)$.  In 
consequence, $y\in (f(C))'$.  Hence, $f(C)$ is a perfect set in $\mathbb{R}$.  
On the other hand, since $C\neq \emptyset$, we see that $f(C)\neq \emptyset$. 
Finally, we will show that $\text{int}(f(C))=\emptyset$.  We suppose, by 
contradiction, that there exists $y\in\text{int}(f(C))$. Then, there is 
$r>0$ such that $(y-r,y+r)\subset f(C)$. Since $(y-r,y+r)$ is a connected 
set, we have that $f^{-1}((y-r,y+r))\subset C \subset A$ is also a connected set.
Let us take $u,v\in f^{-1}((y-r,y+r))$ such that $u<v$. Then, 
$(u,v) \subset f^{-1}((y-r,y+r))\subset C$.  Thus, 
$\text{int}(C) \neq \emptyset$, which is a contradiction. 
Hence, $\text{int}(f(C))=\emptyset$.  Since $f(C)$ is a nonempty perfect 
and nowhere dense set in $\mathbb{R}$, we conclude that $f(C)$ is a Cantor set.
\end{proof}
It is worth mentioning that ``every uncountable $G_\delta$ or $F_\sigma$ set 
in a Polish space contains a homeomorphic copy of the Cantor space'' 
(\cite{Kechris}).  In addition, P. S. Alexandroff (\cite{Alex}) showed that 
every uncountable Borel-measurable set contains a perfect set. 
By using these last facts, one can also obtain some of the 
results of this paper.   However, in Section \ref{SectionL}, we will 
mainly proceed in a different way, by constructing an uncountable perfect 
and nowhere dense subset of the Liouville numbers.  It deserves remark that 
Bendixson's Theorem, which states that every closed subset of the real line 
can be represented as a disjoint union of a perfect set and a countable set,  
is used in the proofs of the main results given in Sections  
\ref{SectionD} and \ref{SectionC}.

This paper is organized as follows. In Section \ref{SectionL}, the existence of 
a Cantor set in the set of Liouville numbers is proved.  Section \ref{SectionD} 
is devoted to show the existence of a Cantor set inside the set of 
Diophantine real numbers. Finally, in Section \ref{SectionC}, a necessary and 
sufficient condition for the existence of a Cantor set contained in 
a subset of $\mathbb{R}$ is given.

%
\section{Existence of a Cantor set contained in the Liouville Numbers} \label{SectionL}
We begin this section showing the existence of an uncountable closed set, $S$, 
contained in $\mathbb{L}$.  Then, we prove that $S$ is a perfect set.  Finally, since 
$S$ is a closed set and $\lambda(\mathbb{L})=0$, where $\lambda$ is the Lebesgue 
measure on the real line, we conclude that $S$ is also nowhere dense.

First, let us consider the following set 
\begin{equation}\label{eq:A}
   A=\{x=(x_n)_{n\in\mathbb{N}}\in \{0,1\}^{\mathbb{N}}: x_{2n}+x_{2n+1}=1, 
	 \quad\forall n\in\mathbb{N}\}.
\end{equation}
The next result concerns the cardinality of set $A$.
\begin{Lemma}\label{cardA}
The set $A$, given in (\ref{eq:A}), has the cardinality of the continuum.
\end{Lemma}
\begin{proof}
Let $f$ be the function given by 
\begin{equation*}
\begin{array}{cccc}
   f:&A  &\longmapsto & \{0,1\}^{\mathbb{N}}\\
	   &x  &\longmapsto     & y,
 \end{array}
\end{equation*}
where $y=(y_n)_{n\in\mathbb{N}}\in \{0,1\}^{\mathbb{N}}$ is defined by
\begin{equation*}
 y_n=\left\{
   \begin{array}{ll}
     1, & \hbox{if }x_{2n+1}=1, \\
     0, & \hbox{if }x_{2n}=1,
   \end{array}\right.
\end{equation*}
for all $n\in \mathbb{N}$. From the definition of function $f$, one gets 
that $f$ is bijective. Then, $|A|=|\{0,1\}^{\mathbb{N}}|$. Since 
$c = |\mathbb R|=|\{0,1\}^{\mathbb{N}}|$, we conclude that $A$ has the cardinality of 
the continuum. 
\end{proof}
Using (\ref{eq:A}), we define the set
\begin{equation}\label{eq:S}
  S=\left\{\sum_{n=1}^{+\infty}\frac{x_{n-1}}{10^{n!}}
	:x=(x_n)_{n\in\mathbb{N}}\in A\right\}.
\end{equation}
The following lemma will be used in the proof of Proposition \ref{P1}.
\begin{Lemma}\label{unicidad}
Let $z=(z_i)_{i\in\mathbb{Z}^+}$ be a sequence such that 
$z_i\in\{-1,0,1\}$ for all $i\in \mathbb{Z}^+$. If
\begin{equation*}
  \sum_{i=1}^{+\infty}\frac{z_i}{10^{i!}}=0,
\end{equation*}
then $z_i=0$ for all $i\in\mathbb{Z}^+$.
\end{Lemma}
\begin{proof}
Since  
\begin{equation*}
   -\frac{z_1}{10}=\sum_{i=2}^{+\infty}\frac{z_i}{10^{i!}},
\end{equation*}
we see that
\begin{align*}
  \left|\frac{z_1}{10}\right| &=  \left|\sum_{i=2}^{+\infty}
	                                \frac{z_i}{10^{i!}}\right|
																	\leq \sum_{i=2}^{+\infty}\frac{|z_i|}{10^{i!}}
                                  \leq\sum_{i=2}^{+\infty}\frac{1}{10^{i!}} \\
															&<	\sum_{i=2!}^{+\infty}\frac{1}{10^{i}} 
															    =\frac{\frac{1}{10^2}}{1-\frac{1}{10}}
																	=\frac{1}{90}.
\end{align*}
Thus, $|z_1|<\frac{1}{9}<1$. Hence, we conclude that $z_1=0$.
\noindent
Now, we suppose, by induction, that for $n\in \mathbb{Z}^+$ 
such that $n\geq 2$, we have that $z_k=0$ for $k \in\{1,2,\ldots,n-1\}$. 
From the fact that
\begin{equation*}
  -\frac{z_n}{10^{n!}}=\sum_{i=n+1}^{+\infty}\frac{z_i}{10^{i!}},
\end{equation*}
we get
\begin{align*}
  \left|\frac{z_n}{10^{n!}}\right| &=\left|\sum_{i=n+1}^{+\infty}
	                                   \frac{z_i}{10^{i!}}\right|
	                                   \leq\sum_{i=n+1}^{+\infty}\frac{|z_i|}{10^{i!}}
																		 \leq\sum_{i=n+1}^{+\infty}\frac{1}{10^{i!}}  \\
                                   &<\sum_{i=(n+1)!}^{+\infty}\frac{1}{10^{i}}
																	  =\frac{\frac{1}{10^{(n+1)!}}}{1-\frac{1}{10}}
																		=\frac{10}{9} \cdot \frac{1}{10^{(n+1)!}}.
\end{align*}
Then, 
\begin{equation*}
   |z_n|<\frac{10}{9} \cdot \frac{10^{n!}}{10^{(n+1)!}}
	 =\frac{10}{9} \cdot 10^{n!(1-(n+1))}=\frac{10}{9} \cdot 10^{-n(n!)}
	 \leq \frac{10}{9} \cdot 10^{-4}=\frac{1}{9000}<1,
\end{equation*}
where we have used the fact that $-n(n!)\leq-4$ for $n\in \{2,3,\ldots\}$. Hence,
$|z_n|<1$, and thus we conclude that $z_n=0$.
\end{proof}
The next proposition shows that the set $S$ has the cardinality of the continuum. 
\begin{Proposition}\label{P1}
Let $S$ and $A$ be the sets respectively defined by (\ref{eq:S}) and (\ref{eq:A}).  
Then, $|S|=|A|=c$.
\end{Proposition}
\begin{proof}
Let $g$ be the function given by 
\begin{equation*}
\begin{array}{cccc}
   g:&A  &\longmapsto & S\\
	   &x=(x_n)_{n \in \mathbb{N}}  &
		 \longmapsto & \displaystyle \sum_{n=1}^{+\infty}\frac{x_{n-1}}{10^{n!}}.
 \end{array}
\end{equation*}
By the definition of function $g$, we see that $g$ is surjective. 
Moreover, by Lemma~\ref{unicidad}, it follows that $g$ is injective. Therefore, 
we conclude that $g$ is bijective.  Hence, $|S|=|A|=c$.
\end{proof}
The subsequent result states that $S$ is closed.
\begin{Proposition}\label{P2}
The set $S \subset \mathbb{R}$, given in (\ref{eq:S}), is a closed subset of $\mathbb{R}$.
\end{Proposition}
\begin{proof}
Let $y\in \mathbb{R}$ be such that there is a sequence 
$(x^n)_{n\in\mathbb{N}}$ in $S$ with $\displaystyle \lim_{n\to+\infty} x^n =y$.
So, for every $n\in\mathbb{N}$, we can write   
\begin{equation*}
   x^n=\sum_{i=1}^{+\infty}\frac{x^n_{i-1}}{10^{i!}},
\end{equation*}
where $(x^n_i)_{i\in\mathbb{N}} \in A$. Since $(x^n)_{n\in\mathbb{N}}$ is a 
convergent sequence, we have that $(x^n)_{n\in\mathbb{N}}$ is a Cauchy sequence. 
Thus, for $\varepsilon_0=\frac{1}{90}>0$, there is 
$N_0\in\mathbb{N}$ such that for all $m,n \in \mathbb{N}$, 
\begin{equation*}
    m,n\geq N_0 \Longrightarrow
   |x^m-x^n|=\left|\sum_{i=1}^{+\infty}
	           \frac{x^m_{i-1}-x^n_{i-1}}{10^{i!}}\right|<\varepsilon_0.
\end{equation*}
Then, for $m,n \in \mathbb{N}$ such that $m,n \geq N_0$, we get
\begin{align*}
    \frac{|x^m_0-x^n_0|}{10} & 
		\leq \left|\sum_{i=1}^{+\infty}\frac{x^m_{i-1}-x^n_{i-1}}{10^{i!}}\right|
		+\left|\sum_{i=2}^{+\infty}\frac{x^m_{i-1}-x^n_{i-1}}{10^{i!}}\right| \\
    & <\varepsilon_0 +\sum_{i=2}^{+\infty}\frac{|x^m_{i-1}-x^n_{i-1}|}{10^{i!}}
		  \leq \varepsilon_0+\sum_{i=2}^{+\infty}\frac{1}{10^{i!}}
			<\varepsilon_0 +\sum_{i=2!}^{+\infty}\frac{1}{10^{i}} \\
    & =\varepsilon_0+\frac{0.1^2}{1-0.1}
		  =\frac{1}{90}+\frac{1}{90}=\frac{1}{45}.
\end{align*}
Therefore, $x^n_0=x^m_0$. Hence, there is $x_0\in\{0,1\}$ such that 
$\displaystyle \lim_{n\to+\infty} x^n_0 = x_0$.
We proceed now by induction on $k$. Let $k\in\mathbb{Z}^+$. Let 
us suppose that for all $j \in\{0, \ldots,k-1\}$, there exist 
$\displaystyle \lim_{n\to+\infty} x^n_j =: x_j \in \{0,1\}$.  Using again the 
fact that $(x^n)_{n\in\mathbb{N}}$ is a Cauchy sequence, we see that for   
$\varepsilon_k=\frac{10}{9} \cdot \frac{1}{10^{(k+2)!}}>0$, there is 
$\widetilde{N}_{k-1}\in\mathbb{N}$ such that for $m,n \in \mathbb{N}$,
\begin{equation}\label{eq1}
   m,n\geq \widetilde{N}_{k-1} \Longrightarrow
   |x^m-x^n|=\left|\sum_{i=1}^{+\infty}\frac{x^m_{i-1}-x^n_{i-1}}{10^{i!}}\right|
	 <\varepsilon_k.
\end{equation}
Since for every $j\in \{0,1,\ldots, k-1\}$, there exist 
$\displaystyle \lim_{n\to+\infty} x^n_j = x_j \in \{0,1\}$ we have that 
there is $N_j\in\mathbb{N}$ such that for all $n \in \mathbb{N}$,
\begin{equation}\label{eq2}
  n \geq N_j \Longrightarrow x^n_j = x_j.
\end{equation}
Let $N_k:= \max \{N_0, N_1, \ldots, N_{k-1}, \widetilde{N}_{k-1} \} \in \mathbb{N}$.
Thus, for all $m, n \in \mathbb{N}$,
\begin{equation}\label{eq3}
  m,n \geq N_k \Longrightarrow 
	|x^m-x^n|=\left|\sum_{i=k+1}^{+\infty}\frac{x^m_{i-1}-x^n_{i-1}}{10^{i!}}\right|
	<\varepsilon_k,
\end{equation}
where in the last inequality we have used (\ref{eq1}) and (\ref{eq2}).  
Then, for all $m, n \in \mathbb{N}$,
\begin{align*}
  m,n \geq N_k \Longrightarrow 
  \frac{|x^m_k-x^n_k|}{10^{(k+1)!}}
	& \leq \left|\sum_{i=k+1}^{+\infty}\frac{x^m_{i-1}-x^n_{i-1}}{10^{i!}}\right|
	  +\left|\sum_{i=k+2}^{+\infty}\frac{x^m_{i-1}-x^n_{i-1}}{10^{i!}}\right| \\
  & <\varepsilon_k + \sum_{i=k+2}^{+\infty}\frac{|x^m_{i-1}-x^n_{i-1}|}{10^{i!}} \\
	& \leq \varepsilon_k +\sum_{i=k+2}^{+\infty}\frac{1}{10^ {i!}} 
	  <\varepsilon_k+\sum_{i=(k+2)!}^{+\infty}\frac{1}{10^{i}} \\
  &  =\varepsilon_k+\frac{0.1^{(k+2)!}}{1-0.1}
	  =\frac{10}{9} \cdot \frac{1}{10^{(k+2)!}}+\frac{10}{9} \cdot \frac{1}{10^{(k+2)!}} \\
	& =\frac{20}{9} \cdot \frac{1}{10^{(k+2)!}},
\end{align*}
where in the second inequality on the right-hand side of the implication above, 
we have used (\ref{eq3}). Thus, for all $m, n \in \mathbb{N}$,
\begin{align}\label{eq4}
  m,n \geq N_k & \Longrightarrow |x^m_k-x^n_k| 
	  < \frac{20}{9} \cdot \frac{10^{(k+1)!}}{10^{(k+2)!}} 
	  = \frac{20}{9} \cdot \frac{1}{10^{(k+1)!(k+1)}} 
    < \frac{20}{9 \cdot 10^4} <1 \nonumber \\
  & \Longrightarrow  x^m_k = x^n_k.
\end{align}
It follows from (\ref{eq4}) that there is $x_k\in\{0,1\}$ 
such that $\displaystyle \lim_{n\to+\infty} x^n_k = x_k$. By the principle 
of finite induction, we conclude that for all $l \in \mathbb{N}$, 
\begin{equation}\label{eq5}
   \displaystyle \lim_{n\to+\infty} x^n_l=:x_l \in \{0,1\}.
\end{equation}
Moreover, it follows from the last expression that for all $l \in \mathbb{N}$, 
there exists $N_l \in \mathbb{N}$ such that for all $n \in \mathbb{N}$,
\begin{equation}\label{eq4a}
  n \geq N_l  \Longrightarrow x^n_l = x_l.
\end{equation}	
\underline{Claim 1}: $x:=(x_i)_{i\in \mathbb{N}} \in A$. \\
Let $i \in \mathbb{N}$.  We see that for all $n \in \mathbb{N}$, 
\begin{equation*}
  x^n_{2i} + x^n_{2i+1} =1.
\end{equation*}
By using (\ref{eq5}) into the last expression we get
\begin{equation}\label{eq6}
   x_{2i} + x_{2i+1} = 1.
\end{equation}
By (\ref{eq5}), $x \in 2^{\mathbb{N}}$, and using (\ref{eq6}), we 
conclude that $x \in A$. \\
\underline{Claim 2}: $\displaystyle y=\sum_{i=1}^{+\infty}\frac{x_{i-1}}{10^{i!}}$. \\
In fact, let $\varepsilon>0$. Since 
$\displaystyle \sum_{i=1}^{\infty}\frac{1}{10^{i!}} < + \infty$, there is 
$a=a(\varepsilon)\in \{2,3, \ldots\}$ such that
\begin{equation}\label{eq7}
    \sum_{i=a}^{+\infty}\frac{1}{10^{i!}}<\varepsilon.
\end{equation}
By (\ref{eq4a}), there exists $P=P(\varepsilon) \in \mathbb{N}$, 
such that for all $n \in \mathbb{N}$,
\begin{equation}\label{eq8}
  n \geq P \Longrightarrow x^n_k =x_k \in \{0,1\}, 
	\forall k \in \{0,1, \ldots, a-2 \}.
\end{equation}
So, for all $n \in \mathbb{N}$,
\begin{align}\label{eq9}
  n \geq P \Longrightarrow  
	\left|x^n - \sum_{i=1}^{+\infty}\frac{x_{i-1}}{10^{i!}}\right| 
	& = \left|\sum_{i=1}^{+\infty}\frac{x^n_{i-1}}{10^{i!}}
	    -\sum_{i=1}^{+\infty}\frac{x_{i-1}}{10^{i!}}\right| \nonumber \\
	&	=	\left|\sum_{i=1}^{a-1}\frac{x^n_{i-1}-x_{i-1}}{10^{i!}}
	    -\sum_{i=a}^{+\infty}\frac{x^n_{i-1}-x_{i-1}}{10^{i!}}\right| \nonumber \\
	& =	\left|\sum_{i=a}^{+\infty}\frac{x^n_{i-1}-x_{i-1}}{10^{i!}}\right| \nonumber \\	
	& \leq \sum_{i=a}^{+\infty}\frac{|x^n_{i-1}-x_{i-1}|}{10^{i!}} 
	  \leq \sum_{i=a}^{+\infty}\frac{1}{10^{i!}} < \varepsilon,	
\end{align}
where in the third equality above we have used (\ref{eq8}) and in the 
last inequality we have used (\ref{eq7}).  Then, 
\begin{equation*}
  \lim_{n \to + \infty} x^n 
	= \sum_{i=1}^{+\infty}\frac{x_{i-1}}{10^{i!}}. 
\end{equation*}
By the uniqueness of the limit in the real line, we conclude that 
$y= \displaystyle \sum_{i=1}^{+\infty}\frac{x_{i-1}}{10^{i!}}$.  \\
Using Claims 1 and 2, we obtain that $y\in S$.  Hence, 
$S$ is a closed subset of $\mathbb{R}$.
\end{proof}
By using a standard proof, we now show the following proposition.
\begin{Proposition}\label{P3}
The set $S$, given in (\ref{eq:S}), is contained in the set of Liouville numbers, 
more precisely 
\begin{equation}\label{subset}
   S\subset (0,1) \cap \mathbb{L} \subset \mathbb{L}.
\end{equation}
\end{Proposition}
\begin{proof}
Let $y \in S$.  Thus, there exists $(x_n)_{n \in \mathbb{N}} \in A$ such that 
$y= \displaystyle \sum_{i=1}^{+\infty}\frac{x_{i-1}}{10^{i!}}$.  Then,
\begin{equation*}
  0 < y = \sum_{i=1}^{+\infty}\frac{x_{i-1}}{10^{i!}} 
	  < \sum_{i=1}^{+\infty}\frac{1}{10^{i!}}
		< \sum_{i=1}^{+\infty}\frac{1}{10^i} 
		=\frac{1}{9} < 1.
\end{equation*}
We now consider $n \in \mathbb{Z}^+$.  We define $q_n, p_n \in \mathbb{Z}$ 
as follows
\begin{equation*}
  q_n := 10 ^{n!} > 1   \quad \text{ and }  \quad 
	p_n := q_n \cdot \sum_{k=1}^n \frac{x_{k-1}}{10^{k!}}.
 \end{equation*}
Therefore,
\begin{align*}
  \left| y - \frac{p_n}{q_n} \right| 
	& = \sum_{k=n+1}^{+\infty}  \frac{x_{k-1}}{10^{k!}} 
	  < \sum_{k=n+1}^{+\infty}  \frac{1}{10^{k!}}      \\
	& <   \sum_{k=(n+1)!}^{+\infty}  \frac{1}{10^k}  
	  = \frac{1}{10^{(n+1)!}} \cdot \sum_{k=0}^{+\infty}  \frac{1}{10^k} \\
	& = \frac{1}{10^{(n+1)!}} \cdot \frac{10}{9} < \frac{10}{10^{(n+1)!}} \\
	& \le \frac{10^{n!}}{10^{(n+1)!}} = \frac{1}{10^{n!(n+1) -n!}} 
	  = \frac{1}{10^{n! \cdot n}} = \frac{1}{q_n^n}.
\end{align*}
For the sake of completeness, we will show here that 
$y \in \mathbb{R} \setminus \mathbb{Q}$.  In fact, we suppose, by contradiction, 
that there are $p, q \in \mathbb{Z}^+$ such that 
$\frac{p}{q} = y = \displaystyle \sum_{i=1}^{+\infty}\frac{x_{i-1}}{10^{i!}}$.  
Since $y \in (0,1)$, we see that $0<p<q$.  Then, 
$p \in \{1, 2, \ldots, q-1 \}$.  Moreover, there is $m \in \mathbb{Z}^+$ 
such that 
\begin{equation}\label{eq10}
  q < 10^{m! \cdot m -1}.
\end{equation}	
Furthermore, the expression    
$\frac{p}{q} = \displaystyle \sum_{i=1}^{+\infty}\frac{x_{i-1}}{10^{i!}}$
is equivalent to 
\begin{equation}\label{eq11}
  p \cdot 10^{m!} = q \cdot \sum_{k=1}^m x_{k-1}  10^{m!-k!} 
	+ q \cdot 10^{m!} \cdot \sum_{k=m+1}^{+\infty} \frac{x_{k-1}}{10^{k!}}.
\end{equation}
Using (\ref{eq11}) we see that 
$\left( \displaystyle q \cdot 10^{m!} \cdot \sum_{k=m+1}^{+\infty} 
\frac{x_{k-1}}{10^{k!}} \right)  \in \mathbb{Z}^+$.  Then,
\begin{align*}
 1 & \le q \cdot 10^{m!} \cdot \sum_{k=m+1}^{+\infty} \frac{x_{k-1}}{10^{k!}} 
     < q \cdot 10^{m!}\cdot \sum_{k=m+1}^{+\infty} \frac{1}{10^{k!}} \\
	 & < q \cdot 10^{m!}\cdot \sum_{k=(m+1)!}^{+\infty} \frac{1}{10^k}
	   = \frac{q \cdot 10^{m!}}{10^{(m+1)!}} \cdot \sum_{k=0}^{+\infty} \frac{1}{10^k} \\
	 & = \frac{q}{10^{m! \cdot m}} \cdot \frac{10}{9} 
	   < \frac{q}{10^{m! \cdot m-1}} <1,
\end{align*}
where in the last inequality we have used (\ref{eq10}).  Last expression 
shows that the assumption $y \in \mathbb{Q}$ leads to a contradiction.  Hence, 
$S\subset (0,1) \cap \mathbb{L} \subset \mathbb{L}$.
\end{proof}
The succeeding result says that the set $S$ is equal to its set of 
accumulation points.
\begin{Proposition}\label{P4}
The set $S$, given in (\ref{eq:S}), is a perfect set, i.e. $S=S'$.
\end{Proposition}
\begin{proof}
Since $S$ is closed, it is enough to show that every element of $S$ is 
an accumulation point of $S$.  In order to prove the last assertion, let $a \in S$,
and $\varepsilon > 0$.  We will show that there exists $b \in S$ such 
that $0 < |a-b|<\varepsilon$.  We take $N \in \mathbb{N}$ satisfying  
$N > \frac{1}{2} \cdot \log_{10} \left(\frac{2}{\varepsilon} \right)$.  
Since $a \in S$, there is $(x_n)_{n \in \mathbb{N}} \in A$ such that 
$a= \displaystyle \sum_{i=1}^{+\infty}\frac{x_{i-1}}{10^{i!}}$.  For all 
$i \in \mathbb{N}$, we define
\begin{equation}\label{eq12}
 y_i:=\left\{
   \begin{array}{ll}
     1-x_i, & \hbox{if } (2N-i)(2N+1-i) = 0, \\
     x_i, & \hbox{otherwise}.
   \end{array}\right.
\end{equation}
Since $(x_n)_{n \in \mathbb{N}} \in \{0,1\}^{\mathbb{N}}$, it follows 
directly from (\ref{eq12}) that  $(y_n)_{n \in \mathbb{N}} \in \{0,1\}^{\mathbb{N}}$.  
We will now show that for all $i \in \mathbb{N}$, $y_{2i} + y_{2i+1} =1$.  In fact, 
if $i=N$, then $y_{2i} + y_{2i+1} = 1-x_{2i}+1-x_{2i+1}=1+1-1=1$.  On the 
other hand, if $i \neq N$, then  $y_{2i} + y_{2i+1} = x_{2i} + x_{2i+1} =1$.  
Thus, $(y_n)_{n \in \mathbb{N}} \in A$, and therefore  
\begin{equation}\label{eq13}
  b:= \sum_{i=1}^{+\infty} \frac{y_{i-1}}{10^{i!}} \in S.
\end{equation}
In addition,
\begin{align*}
 0< |a-b| & = \left| \sum_{i=1}^{+\infty} \frac{x_{i-1}}{10^{i!}} 
           - \sum_{i=1}^{+\infty} \frac{y_{i-1}}{10^{i!}} \right| 
				 = \left|  \sum_{i=1}^{+\infty} \frac{x_{i-1}-y_{i-1}}{10^{i!}} \right| \\
			 & = \left|  \sum_{i=0}^{+\infty} \frac{x_i-y_i}{10^{(i+1)!}} \right| 
			   = \left| \frac{x_{2N}-y_{2N}}{10^{(2N+1)!}} 
			     + \frac{x_{2N+1}-y_{2N+1}}{10^{(2N+2)!}} \right|	 \\
			 & = \left| \frac{x_{2N}-1 +x_{2N}}{10^{(2N+1)!}} 
			     + \frac{x_{2N+1}-1 + x_{2N+1}}{10^{(2N+2)!}} \right|	\\
			 & = \left| \frac{2x_{2N}-1}{10^{(2N+1)!}} 
		     + \frac{2x_{2N+1}-1}{10^{(2N+2)!}}\right|		
			  \le  \frac{|2x_{2N}-1|}{10^{(2N+1)!}} + \frac{|2x_{2N+1}-1|}{10^{(2N+2)!}} \\
			 & = \frac{1}{10^{(2N+1)!}}  + \frac{1}{10^{(2N+2)!}}	
			  \le \frac{1}{10^{(2N)!}}  + \frac{1}{10^{(2N)!}}	\\
			& \le \frac{2}{10^{2N}} < \varepsilon,
\end{align*}
where the last equality is a consequence of the fact that for all  $z \in \{0,1\}, |2z-1|=1$.  
This concludes the proof of the proposition.
\end{proof}
We now proceed to prove the key theorem of this section.
\begin{Theorem}\label{ThL}
The set $S$, given in (\ref{eq:S}), is a Cantor set contained in the set of 
Liouville numbers.
\end{Theorem}
\begin{proof}
Let $S$ be the set given by (\ref{eq:S}).  By Propositions \ref{P1}, \ref{P3} 
and \ref{P4}, $S$ is an uncountable perfect set contained in $\mathbb{L}$.  Moreover, 
by Proposition \ref{P2}, $S$ is a closed subset of $\mathbb{R}$, since  
$\lambda(\mathbb{L})=0$ and $S\subset \mathbb{L}$, we have that $S$ is a nowhere 
dense subset of $\mathbb{R}$. We may therefore conclude that $S$ is a Cantor 
set such that $S\subset \mathbb{L}$.
\end{proof}
Before ending this section, we state an important definition and a 
lemma that we will use in the proof of Proposition \ref{P5} below.  
\begin{Definition}[Cantor-Bendixson's derivative]\label{def:D_CB}
Let $A$ be a subset of a topological space. For a given ordinal number 
$\alpha \in \mathbf{OR}$, we define, using transfinite recursion, the  
$\alpha$-th derivative of $A$, written $A^{(\alpha)}$, as follows:
\begin{itemize}
  \item $A^{(0)}=A$,
  \item $A^{(\beta+1)}=(A^{(\beta)})'$, for all ordinal $\beta$,
  \item $\displaystyle A^{(\lambda)}=\bigcap_{\gamma<\lambda} A^{(\gamma)}$, 
	for all limit ordinal $\lambda\neq 0$.
\end{itemize}
\end{Definition}
The next lemma and its proof can be found in \cite[Lemma 2.1]{BAS-AM}.
\begin{Lemma}\label{derivative}
Suppose that $n\in \mathbb{Z}^+$. Let $F_1, F_2, \ldots, F_n$ be closed subsets 
of $\mathbb{R}$.  Then, for all ordinal number $\alpha \in \mathbf{OR}$, 
we have that 
\begin{equation*}
	\left(\bigcup_{k=1}^n F_k\right)^{(\alpha)} = \bigcup_{k=1}^n F_k^{(\alpha)}.
\end{equation*}
\end{Lemma}
We close this section with a general topological result on the real line.
\begin{Proposition}\label{P5}
Every element of a perfect set $C \subset \mathbb{R}$ is a condensation point of $C$.
\end{Proposition}
\begin{proof}
Let $C \subset \mathbb{R}$ be a perfect set.  We suppose, for the sake of contradiction, 
that there is $a \in C$ such that $a$ is not a condensation point of $C$.  Then, 
there exists $r > 0$ such that $C \cap (a-r,a+r)$ is countable.  Thus, 
$C \cap [a-r, a+r]$ is also countable.  Since $C$ is a perfect set, we have 
that $C$ is closed.  Hence, $C \cap [a-r, a+r]$ is a closed and countable 
subset of $\mathbb{R}$. By Theorem C of \cite{Cantor}, there exists a countable 
ordinal number $\alpha \in \Omega$ such that the $\alpha$-th derivative of 
$C \cap [a-r, a+r]$ is empty, namely, $(C \cap [a-r, a+r])^{(\alpha)} = \emptyset$.  
Moreover, we write
\begin{align*}
 C & = (C \setminus (a-r, a+r)) \biguplus (C \cap (a-r, a+r)) \\
   & \subset (\mathbb{R} \setminus (a-r, a+r)) \cup (C \cap [a-r, a+r]),
\end{align*}
where $\mathbb{R} \setminus (a-r, a+r) = (-\infty, a-r] \biguplus [a+r, +\infty)$ is 
a perfect set.  Using Lemma~\ref{derivative}, we see that
\begin{align*}
 a \in C  = C^{(\alpha)} & \subset \left [(\mathbb{R} \setminus (a-r, a+r)) 
                           \cup (C \cap [a-r, a+r]) \right]^{(\alpha)} \\
                         & = (\mathbb{R} \setminus (a-r, a+r))^{(\alpha)} 
                           \cup (C \cap [a-r, a+r])^{(\alpha)} \\
                         & =  \mathbb{R} \setminus (a-r, a+r),   
\end{align*}
which is a contradiction.  Therefore, every element of $C$ is a condensation 
point of $C$.
\end{proof}
%
\section{Existence of a Cantor set contained in the Diophantine Numbers}\label{SectionD}
First, let us consider the following representation of  the set of Liouville numbers, 
\begin{equation*}
     \mathbb{L}=\bigcap_{n\in\mathbb{N}}U_n,
\end{equation*}
where
\begin{equation*}
   U_n=\bigcup_{q=2}^{+\infty}\bigcup_{p\in\mathbb{Z}}
	     \left[\left(\frac{p}{q}-\frac{1}{q^n},\frac{p}{q}\right)
			 \biguplus \left(\frac{p}{q},\frac{p}{q}+\frac{1}{q^n}\right)\right] 			 
\end{equation*}
is an open and dense set of $\mathbb{R}$, for all $n\in\mathbb{N}$.  Then,
\begin{equation}\label{eq20}
  \mathbb{D}=\mathbb{R} \setminus \mathbb{L} = \mathbb{L}^c 
	          =\left(\bigcap_{n\in\mathbb{N}}U_n\right)^c
						=\bigcup_{n\in\mathbb{N}}U_n^c=\bigcup_{n\in\mathbb{N}}D_n,
\end{equation}
where $D_n=U_n^c$ is a closed and nowhere dense set, for all $n\in\mathbb{N}$. 
We come now to the main result of this section.
\begin{Theorem}\label{ThD}
There is a set $C\subset \mathbb{D}$ such that $C$ is a Cantor set.
\end{Theorem}
\begin{proof}
Since $\lambda(\mathbb{L})=0$, $\mathbb{D}$ is an uncountable set. Using 
(\ref{eq20}), we see that there is a $k\in\mathbb{N}$ such that $D_k$ is 
an uncountable set.  Let $C$ be the set of all condensation points of $D_k$.
By Bendixson's Theorem, $D_k= (D_k \setminus C) \biguplus C$, where 
$D_k \setminus C$ is countable and $C$ is a perfect set.  We see that 
$C$ is an uncountable subset of $\mathbb{R}$.  Also, since $D_k$ is 
a nowhere dense set, we get that $C$ is also nowhere dense.  Hence, $C$ 
is a Cantor set contained in $\mathbb{D}$.
\end{proof}
%
\section{A necessary and sufficient condition for the existence of a Cantor set 
contained in a subset of the real line}\label{SectionC}
We begin this section with a preliminary result.
\begin{Lemma}\label{CP}
Every nonempty perfect set in $\mathbb{R}$ contains a Cantor set.
\end{Lemma}
\begin{proof}
Let $P\subset\mathbb{R}$ be a nonempty perfect set. There are two cases 
to consider.
\begin{itemize}
 \item If $\text{int}(P)=\emptyset$, then $P$ is a Cantor set.
 \item If $a\in \text{int}(P)$, there is $r>0$ such that $(a-r,a+r)\subset P$.
   Let $\widetilde{C}$ be the usual triadic Cantor set in the closed interval 
	 $\left[a-\frac{r}{3},a+\frac{r}{3}\right]$. Then, 
   \begin{equation*}
     \widetilde{C}\subset \left[a-\frac{r}{3},a+\frac{r}{3}\right]
	    \subset (a-r,a+r)
			\subset P.
   \end{equation*}
\end{itemize}
\end{proof}
Finally, we show the purpose of this section.
\begin{Theorem}
$X\subset \mathbb{R}$ contains a Cantor set if and only if 
$X$ contains a closed and uncountable subset of $\mathbb{R}$.
\end{Theorem}
\begin{proof}
Let $F\subset X$ be a closed and uncountable subset of the real line. 
By Bendixson's Theorem, there is a perfect and uncountable set $P\subset F$. 
By Lemma~\ref{CP}, there is a Cantor set $K$, such that 
$K\subset P \subset F \subset X$.  Reciprocally, since all Cantor sets are closed  
and uncountable, the theorem is proved.
\end{proof}
%

\vspace{1cm}


\begin{thebibliography}{13.}
\bibitem{Alex} P. S. Alexandroff, \emph{Sur la puissance des ensembles mesurables}, 
        C. R. Acad. Sci. Paris {\bf{162}} (1916), 323–325.
\bibitem{BAS-AM} B. \'Alvarez-Samaniego	and A. Merino, \emph{A primitive associated 
         to the Cantor-Bendixson derivative on the real line}, J. Math. Sci. Adv. Appl. 
				 {\bf{41}} (2016), 1-33.
\bibitem{Cantor} G. Cantor, \emph{{Sur divers th\'eor\`emes de la th\'eorie des 
         ensembles de points situ\'es dans un  espace continu \`a $n$ dimensions}}, 
				 Acta Math. {\bf{2}} (1883), 409-414. 
\bibitem{Kechris} A. S. Kechris, \emph{Classical descriptive set theory}, 
        Springer Verlag, Graduate Texts in Mathematics 156, 1995.
\bibitem{Liouville1844a} J. Liouville, \emph{Sur des classes tr\`es \'etendues 
        de quantit\'es dont la valeur n'est ni alg\'ebrique, ni m\^eme r\'eductible 
        \`a des irrationnelles alg\'ebriques}, C. R. Acad. Sci. Paris {\bf{18}} (1844), 883–885.
\bibitem{Liouville1844b} J. Liouville, \emph{Nouvelle d\'emonstration d'un th\'eor\`eme 
        sur les irrationnelles alg\'ebriques ins\'er\'e dans le Compte Rendu de la 
				derni\`ere s\'eance}, C. R. Acad. Sci. Paris {\bf{18}} (1844), 910–911.		
\bibitem{Liouville1851} J. Liouville, \emph{Sur des classes tr\`es \'etendues 
        de quantit\'es dont la valeur n'est ni alg\'ebrique, ni m\^eme r\'eductible 
        \`a des irrationnelles alg\'ebriques}, 
				J. Math. Pures et Appl. {\bf{16}} (1851), 133–142.
\end{thebibliography}
\end{document}